\documentclass[11pt,english]{article}
\usepackage[margin=1.95 cm,bottom=15mm,footskip=7mm,top=15mm]{geometry}
\usepackage{enumerate}
\usepackage{url}
\usepackage{graphicx}
\usepackage{amssymb}
\usepackage{amsthm}
\usepackage{verbatim}
\usepackage{amsmath}
\usepackage{enumerate}
\usepackage{boolexpr, xstring}
\usepackage[table]{xcolor}
\usepackage{enumitem}
\usepackage[numbers]{natbib}
\usepackage{subcaption}
\usepackage{floatrow}
\usepackage{setspace}
\usepackage[bookmarks=true,bookmarksnumbered=false,bookmarksopen=false,pdfborder={0 0 0},pdfborderstyle={},backref=false,colorlinks=false]{hyperref}

\usepackage{appendix}
\usepackage[nameinlink,capitalise,noabbrev]{cleveref}
\crefname{appsec}{Appendix}{Appendices}
\crefformat{equation}{#2(#1)#3}

\theoremstyle{plain}

\newtheorem{theorem}{Theorem}[section]
\newtheorem{definition}[theorem]{Definition}
\newtheorem{proposition}[theorem]{Proposition}
\newtheorem{lemma}[theorem]{Lemma}
\crefname{lemma}{Lemma}{Lemmas}
\newtheorem{claim}[theorem]{Claim}

\newtheorem{observation}[theorem]{Observation}

\global\long\def\hh#1#2{\mathrm{h}_{#2}(#1)}
\global\long\def\vv#1#2{\mathrm{v}_{#2}(#1)}
\newcommand{\height}{\mathrm{h}}
\newcommand{\HT}{\mathrm{HT}}
\newcommand{\htf}{\mathrm{ht}_G}
\newcommand{\hth}{\mathrm{ht}_H}
\newcommand{\htt}{\mathrm{ht}_{G \setminus T}}
\newcommand{\N}{\mathbb N}
\newcommand{\llex}{\prec_{\mathrm{lex}}}
\newcommand{\glex}{\succ_{\mathrm{lex}}}

\global\long\def\E{\mathbb{E}}
\global\long\def\NN{\mathbb{N}}
\global\long\def\one{\boldsymbol{1}}
\global\long\def\floor#1{\left\lfloor #1\right\rfloor }
\global\long\def\Var{\operatorname{Var}}
\global\long\def\ceil#1{\left\lceil #1\right\rceil }

\global\long\def\ordeqE{\le_G}
\global\long\def\ordE{<_G}
\global\long\def\d{\overline{d}}
\global\long\def\ordeqV{\le_G^V}

\newcommand{\leqlex}{\preceq_{\mathrm{lex}}}
\newcommand{\geqlex}{\succeq_{\mathrm{lex}}}

\let\originalleft\left
\let\originalright\right
\renewcommand{\left}{\mathopen{}\mathclose\bgroup\originalleft}
\renewcommand{\right}{\aftergroup\egroup\originalright}

\global\long\def\mk#1{}

\begin{document}

\title{\texorpdfstring{\vspace{-1.5cm}}{}
Nearly-linear monotone paths in edge-ordered graphs}
\date{}
\author{
Matija Buci\'c\thanks{Department of Mathematics, ETH, Z\"urich, Switzerland. Email: \href{mailto:matija.bucic@math.ethz.ch} {\nolinkurl{matija.bucic@math.ethz.ch}}.}
\and
Matthew Kwan\thanks{Department of Mathematics, Stanford University, CA 94305. Email: \href{mailto:mattkwan@stanford.edu} {\nolinkurl{matthew.a.kwan@gmail.com}}. Research supported in part by SNSF project 178493.}
\and
Alexey Pokrovskiy\thanks{Department of Economics, Mathematics, and Statistics, Birkbeck, University of London, U.K. Email: \href{mailto:Dr.Alexey.Pokrovskiy@gmail.com} {\nolinkurl{Dr.Alexey.Pokrovskiy@gmail.com}}.}
\and
Benny Sudakov\thanks{Department of Mathematics, ETH, Z\"urich, Switzerland. Email:
\href{mailto:benjamin.sudakov@math.ethz.ch} {\nolinkurl{benjamin.sudakov@math.ethz.ch}}.
Research supported in part by SNSF grant 200021-175573.}
\and
Tuan Tran\thanks{Department of Mathematics, ETH, Z\"urich, Switzerland. Email:
\href{mailto:manh.tran@math.ethz.ch} {\nolinkurl{manh.tran@math.ethz.ch}}.
Research supported by the Humboldt Research Foundation.}
\and
Adam Zsolt Wagner\thanks{Department of Mathematics, ETH, Z\"urich, Switzerland. Email:
\href{mailto:zsolt.wagner@math.ethz.ch} {\nolinkurl{zsolt.wagner@math.ethz.ch}}.}
}

\maketitle

\begin{abstract}
How long a monotone path can one always find in any edge-ordering of the complete graph $K_n$? This appealing question was first asked by Chv\'atal and Koml\'os in 1971, and has since attracted the attention of many researchers, inspiring a variety of related problems. The prevailing conjecture is that one can always find a monotone path of linear length, but until now the best known lower bound was $n^{2/3-o(1)}$. In this paper we almost close this gap, proving that any edge-ordering of the complete graph contains a monotone path of length $n^{1-o(1)}$.
\end{abstract}

\section{Introduction}

An \emph{edge-ordering} of a graph $G$ is a total order $\ordeqE$ of the edge set $E(G)$. A path $P$ in $G$ is said to be \emph{monotone} if the consecutive edges of $P$ form a monotone sequence with respect to $\ordeqE$.

In 1971, Chv\'atal and Koml\'os~\cite{CK71} asked for the length of the longest monotone path one can guarantee in any edge-ordering of the complete $n$-vertex graph $K_n$. To be precise, let $f(K_n)$ be the maximum $\ell$ such that every edge-ordering of $K_n$ has a monotone path of length $\ell$. What is the value of $f(K_n)$, as a function of $n$? This question was originally motivated by an analogous problem for directed graphs, which is closely related to both of the celebrated theorems of Erd\H os and Szekeres concerning convex subsets of points in the plane and concerning monotone subsequences of sequences of real numbers. Although the question of Chv\'atal and Koml\'os sounds rather innocent, it has turned out to be quite challenging even to understand the order of magnitude of $f(K_n)$, and progress on this problem over the years has been rather sparse.

The first nontrivial results were proved by Graham and Kleitman~\cite{GK73}, who showed that there is always a monotone path of length $\Omega(\sqrt n)$. They also constructed an edge-ordering permitting no monotone path longer than $(3/4)n$. That is, $\Omega(\sqrt n)\le f(K_n)\le (3/4)n$. Apart from some results for small values of $n$ (see~\cite{BCM05}), until recently the only improvements were to the upper bound. First, R\"odl~\cite{rodl73} proved that $f(K_n)\le(2/3+o(1))n$, and then Calderbank, Chung and Sturtevant~\cite{CCS84} improved this to $f(K_n)\le(1/2+o(1))n$ (they also cited an unpublished upper bound of $(7/12+o(1))n$ by Alspach, Heinrich and Graham). In a recent breakthrough, Milans~\cite{mil17} obtained the first new lower bound in almost 50 years. He proved that any edge-ordering of $K_n$ always contains a monotone path of length $\Omega\left((n/\log n)^{2/3}\right)$. In this paper, we almost close the gap between upper and lower bounds on this problem, showing that there exists a nearly-linear monotone path.

\begin{theorem}
\label{thm:altitude-complete}
In any edge-ordering of the complete graph $K_n$, there is a monotone path of length $$f(K_n)\ge \frac{n}{2^{O\left(\sqrt{\log n\log \log n}\right)}}=n^{1-o(1)}.$$
\end{theorem}

Although Chv\'atal and Koml\'os' original question was for complete graphs, it is also natural to ask the analogous question for other graphs. The \emph{altitude} $f(G)$ of a graph $G$ is defined as the maximum $k$ such that every edge-ordering of $G$ has a monotone path of length $k$. R\"odl~\cite{rodl73} proved that if $G$ has average degree $d$ then $f(G)=\Omega(\sqrt d)$. For sufficiently dense graphs, Milans~\cite{mil17} improved this to $f(G)=\Omega(d/{(n^{1/3}(\log n)^{2/3}}))$, where $n$ is the number of vertices in $G$. Of course, the longest path in a graph with average degree $d$ may be as short as $d$ (if $G$ is a disjoint union of cliques of size $d+1$), in which case $f(G)\le d$. In general, it follows from Vizing's theorem that $f(G) \le \Delta(G)+1$ (where $\Delta(G)$ denotes the maximum degree of $G$). This and several other upper bounds depending on various parameters of the graph were obtained in \cite{RSY01} by Roditty, Shoham and Yuster. On the other hand it was proved by Alon~\cite{Alo03} (improving a result of Yuster~\cite{Yus01}) that there are $d$-regular graphs with altitude at least $d$. Here we prove the almost-optimal result that \emph{all} graphs with average degree $d$ have altitude almost as large as $d$, as long as $d$ is not too small.

\begin{theorem}
\label{thm:altitude-d}
Let $G$ be a graph with $n$ vertices and average degree $d\geq 2$. Then $$f(G)\ge \frac{d}{2^{O\left(\sqrt{\log d \log \log n}\right)}}.$$
\end{theorem}

\noindent We remark that the general notion of altitude for non-complete graphs is actually essential for our proof of \cref{thm:altitude-complete}. \cref{thm:altitude-complete} is clearly a special case of \cref{thm:altitude-d}, which itself will be a corollary of a more technical result (\cref{thm:1-eps-regular}) essentially giving a lower bound on $f(G)$ for graphs $G$ satisfying certain conditions.

Apart from the work already mentioned, Chv\'atal and Koml\'os' question has inspired a large number of related problems. Various authors have studied the altitude of specific graphs or classes of graphs, such as planar graphs, trees, hypercubes and random graphs (see for example~\cite{CM05,DMPRT16,MBCFH05,RSY01}). There has also been interest in finding the longest monotone \emph{trail}, rather than path, in edge-ordered graphs (a trail is a walk in a graph which may repeat vertices but not edges). This problem was already proposed in the 1971 paper of Chv\'atal and Koml\'os, and was solved for complete graphs by Graham and Kleitman~\cite{GK73}. Another interesting variant of Chv\'atal and Koml\'os' question considers a \emph{random} edge-ordering of the edges of a graph, instead of the worst-case ordering. The study of this problem for complete graphs was suggested by Lavrov and Loh~\cite{LL16}, who conjectured that with probability $1-o(1)$, the random edge-ordering of the complete graph contains a monotone Hamiltonian path. This was proved by Martinsson~\cite{Mar16}, and the problem of finding the longest monotone trail in the random edge-ordering was recently solved by Angel, Ferber, Sudakov and Tassion~\cite{AFST18}.

\vspace{0.30cm}

\noindent{\bf Structure of the paper and the proof:} The rest of the paper is organised as follows. In \cref{sec:regularisation} we prove a ``regularisation lemma'' showing that graphs with average degree $d$ have a subgraph with minimum and maximum degree close to $d$. We find this lemma to be of independent interest and believe that it might be useful for other applications as well.

In \cref{sec:height-table} we discuss the notion of a \emph{height table} first introduced by Milans. This is a structure that arranges the edges of a graph in a way that is convenient for finding monotone paths. We also prove some crucial lemmas describing how the height table changes after deleting edges and vertices. Results of a similar flavour were proved by Milans, but our results have much more flexibility. These results will be strongest for graphs that are close to regular (which explains why we need the regularisation lemma).

Then, in \cref{sec:main-proof} we present the details of our proof of \cref{thm:altitude-d}. A crucial ingredient which is completely new to our proof is a lemma showing that given a particular edge $e$ in an edge-ordered graph $G$, there is a way to explore the height table to find a dense subgraph of $G$ consisting of edges which are accessible from $e$ via a short monotone trail (then, if we can find a monotone path in this subgraph, we can connect it to $e$). Our proof then takes an iterative approach, inductively finding monotone paths in various subgraphs and stitching them together.

In \cref{sec:concluding} we have some concluding remarks, including a simple observation that sparse random graphs typically have monotone paths of length linear in their average degree. Finally, in \cref{sec:no-regular-subgraph} we present a construction showing that our regularisation lemma is essentially sharp.

\noindent{\bf Notation:} We use standard asymptotic notation throughout. For functions $f=f\left(n\right)$
and $g=g\left(n\right)$ we write $f=O\left(g\right)$ to mean there
is a constant $C$ such that $\left|f\right|\le C\left|g\right|$,
we write $f=\Omega\left(g\right)$ to mean there is a constant $c>0$
such that $f\ge c\left|g\right|$ for sufficiently large $n$, we
write $f=\Theta\left(g\right)$ to mean that $f=O\left(g\right)$
and $f=\Omega\left(g\right)$, and we write $f=o\left(g\right)$ or
$g=\omega\left(f\right)$ to mean that $f/g\to0$ as $n\to\infty$.
All asymptotics are as $n\to\infty$ unless specified otherwise.

We also use standard graph-theoretic notation. For any graph $G$, we denote by $V(G)$ its set of vertices, by $E(G)$ its set of edges, by $\d(G)=2|E(G)|/|G|$ its average degree and by $\Delta(G)$ its maximal degree. In a directed graph we denote the out-degree and in-degree of a vertex $v$ by $d^+(v)$ and $d^-(v)$ respectively.

For a real number $x$, the floor and ceiling functions are denoted
$\floor x=\max\{i\in\mathbb{Z}:i\le x\}$ and $\ceil x=\min\{i\in\mathbb{Z}:i\ge x\}$.  Finally, all logarithms are in base 2, unless stated otherwise.

\section{A regularisation lemma}
\label{sec:regularisation}
In this section we will present a lemma which we will need in the proof of our main result and which we believe might have applications in other situations as well. Given a graph with average degree $d$, this lemma allows us to find an almost-regular subgraph whose average degree is only slightly smaller than $d$. Our lemma is closely related to a conjecture of Erd{\H os} and Sauer, proved by Pyber~\cite{Pyb85}, regarding the existence of regular subgraphs in graphs with many edges (see also~\cite{AFK84,AKS08}).

\begin{lemma}
\label{lem:regularisation}Let $G$ be a graph with $n\ge 2$ vertices
and average degree $d$. Then $G$ has a (possibly non-induced) subgraph
with all degrees lying in the range $\left[d',6d'\right]$, where $d'=\left(\floor{\left(d/4-1\right)/\ceil{\log n}}+1\right)/6$.
\end{lemma}

To prove \cref{lem:regularisation} we follow the approach of Pyber
\cite{Pyb85} used to solve the Erd{\H o}s--Sauer conjecture \cite{ES81}. The heart of the proof is
the following lemma.
\begin{lemma}
\label{lem:pyber-matching}Let $G$ be a bipartite graph with bipartition
$A\cup B$ and minimum degree $\delta$, such that $\left|A\right|\ge\left|B\right|$.
Then, we can find sequences of nonempty subsets $A=A_{0}\supseteq A_{1}\supseteq\dots\supseteq A_{\delta}$
and $B=B_{0}\supseteq B_{1}\supseteq\dots\supseteq B_{\delta}$, and
edge-disjoint matchings $M_{1},\dots,M_{\delta}$, such that $\left|A_{i}\right|=\left|B_{i}\right|$ for each $1\le i\le\delta$, and each $M_{i}$ is a perfect matching
of $G\left[A_{i}\cup B_{i}\right]$.
\end{lemma}

\begin{proof}
We construct these sequences iteratively, with the additional property that $N_{G_i}(A_{i}) \subseteq B_{i}$, where $G_i$ is a graph obtained from $G$ by removing the edges of $M_1,\dots ,M_{i-1}$ (so $G_1=G$). By definition this property is satisfied for $i=0$. Suppose for some $1\le i\le\delta$
that $A_{0},\dots,A_{i-1}$, $B_{0},\dots,B_{i-1}$ and $M_{1},\dots,M_{i-1}$
have already been constructed; we show how to construct $A_{i},B_{i},M_{i}$.
Let $A_{i}\subseteq A_{i-1}$ be a minimal nonempty subset of
$A_{i-1}$ such that $\left|N_{G_{i}}\left(A_{i}\right)\right|\le\left|A_{i}\right|$, and let $B_i=N_{G_i}(A_{i})$. Since $\left|N_{G_{i}}\left(A_{i-1}\right)\right|\le \left|N_{G_{i-1}}\left(A_{i-1}\right)\right|\le\left|B_{i-1}\right|\le \left|A_{i-1}\right|$, the set $A_i$ is well-defined, and because $G_i$ has minimum degree at least $\delta-(i-1)$ we get $|B_i|\ge \delta-(i-1)\ge1$.

We also claim that $\left|B_{i}\right|=\left|A_{i}\right|$. Indeed, if $|B_i|<|A_i|$, then $|A_i|>1$, and deleting any element of $A_i$ would give a smaller nonempty set $A_i'$ satisfying $\left|N_{G_{i}}\left(A_{i}'\right)\right|\le\left|A_{i}'\right|$, which contradicts the minimality of $A_{i}$. Finally, consider the induced subgraph $G_{i}\left[A_{i}\cup B_{i}\right]$.
Observe that by minimality $\left|N_{G_{i}}\left(S\right)\right|\ge\left|S\right|$
for $S\subseteq A_{i}$. So, by Hall's theorem $G_{i}\left[A_{i}\cup B_{i}\right]$ has
a perfect matching $M_{i}$, as desired.
\end{proof}
We will also need two additional (well-known) lemmas. First, we need the fact that any graph $G$ has a bipartite subgraph with at least half as many edges as $G$. See for example \cite[Lemma~1]{Erd65}.
\begin{lemma}
\label{lem:maxcut}Let $G$ be a graph with $n$ vertices and average
degree $d$. Then $G$ has a bipartite subgraph with average degree
at least $d/2$.
\end{lemma}

Next, we need the fact that a graph with average degree $d$ has a subgraph with minimum degree at least $d/2$. See for example~\cite[Proposition~1.2.2]{Die}.
\begin{lemma}
\label{lem:min-degree}Let $G$ be a graph with $n$ vertices and
average degree $d$. Then $G$ has an induced subgraph with minimum
degree at least $d/2$.
\end{lemma}

Now we prove \cref{lem:regularisation}.
\begin{proof}[Proof of \cref{lem:regularisation}]
Let $G'$ be a bipartite subgraph of $G$ with bipartition $A\cup B$
(say $\left|A\right|\ge\left|B\right|$) and minimum degree at least $d/4$.
Such a subgraph exists by \cref{lem:maxcut} and \cref{lem:min-degree}.
Then, apply \cref{lem:pyber-matching} to $G$ to obtain nonempty subsets
$A\supseteq A_{1}\supseteq\dots\supseteq A_{d/4}$ and $B\supseteq B_{1}\supseteq\dots\supseteq B_{d/4}$,
and edge-disjoint matchings $M_{1},\dots,M_{d/4}$, such that each
$M_{i}$ is a perfect matching of $G'\left[A_{i}\cup B_{i}\right]$.

Let $x=\floor{\left(d/4-1\right)/\ceil{\log n}}$. Observe that
there must be some $q$ (satisfying $1\le q\le d/4-x$) such that
$\left|B_{q}\right|\le2\left|B_{q+x}\right|$. Indeed, otherwise we
would have $1+\ceil{\log n}x\le d/4$ and
\[
\left|B_{1}\right|>2^{\ceil{\log n}}\left|B_{1+\ceil{\log n}x}\right|\ge n,
\]
which is impossible. Fix such a $q$, and let $H$ be the graph on the vertex set $A_{q}\cup B_{q+x}$ containing all the edges of the matchings $M_q,\dots,M_{q+x}$ which are contained in $A_{q}\cup B_{q+x}$.
Each of the vertices in $B_{q+x}$ have degree $x+1$ in $H$, and
$\left|A_{q}\cup B_{q+x}\right|=\left|B_{q}\right|+\left|B_{q+x}\right|\le3\left|B_{q+x}\right|$,
meaning that the average degree of $H$ is at least $\left(x+1\right)/3$.
Also, each vertex in $A_{q}$ has degree at most $x+1$ in $H$.
The desired result then follows from \cref{lem:min-degree}.
\end{proof}
One can prove that the assertion of \cref{lem:regularisation} is essentially optimal. As we will not make use of this fact, we defer further details to \cref{sec:no-regular-subgraph}.

\section{Height tables}
\label{sec:height-table}

The notion of a \emph{height table} of a graph was introduced by Milans in \cite{mil17}, and will play a central role in our proof of \cref{thm:altitude-d}. In this section we make some definitions and prove a few lemmas regarding height tables.

Height tables are only uniquely defined for graphs which have an ordering on both their vertices and edges. An \emph{ordered} graph is a graph $G$ equipped with a total order $\ordeqE$ on $E(G)$ and a total order $\ordeqV$ on $V(G)$. We emphasise that the only purpose of the vertex ordering is for height tables to be well-defined, and monotone paths are defined only with respect to the edge-ordering. Given an ordered graph $G$ we define the lexicographic order $\llex$ on $\N\times V(G)$ by $(i,v)\leqlex (i',v')$ if either $i < i'$ or $i=i'$ and $v \ordeqV v'$. In this paper, the natural numbers do not include zero.

\begin{definition}
For an ordered graph $G$, the \emph{height table} of $G$, denoted $\HT(G)$, is a partially filled array indexed by $\N\times V(G)$, constructed as follows. Go through every $(i,v)\in \N\times V(G)$ in increasing order according to $\leqlex$. Let $\HT(G)_{i,v}$ be the $\ordeqE$-largest edge containing $v$ which has not yet been entered into $\HT(G)$. (If all edges containing $v$ have already been entered into $\HT(G)$, then $\HT(G)_{i,v}$ is left blank).
\end{definition}

Notice that every edge of $G$ gets entered exactly once in $\HT(G).$ For an edge $e\in E(G)$, let the \emph{height} of $e$ in $G,$ denoted by $\hh e G,$ be the row of $\HT(G)$ containing $e,$ and let $\vv e G$ be the column of $\HT(G)$ containing $e.$ For an edge $e\in E(G)$ we write $\htf(e)=(\hh e G,\vv e G).$ \mk{I changed $\height(e,G)$ to $\hh e G$ for consistency with $\htf$. I made it a macro in case you want to change it back.} We emphasise that the order $\leqlex$ runs roughly in the \emph{opposite} order to $\ordeqE$: if $e$ is $\ordeqE$-large, then it will tend be entered into the height table early in its construction, meaning that $\htf(e)$ will be $\leqlex$-small.

To see why the notion of a height table is useful for finding monotone paths, it is convenient to think of paths as having a specified direction (therefore we can say a path in an edge-ordered graph is increasing or decreasing). Starting from an edge $xy$ high up in the height table (in column $x=\vv{xy}{G}$, say), we can then look at the edge $yz$ in position $(y,\hh{xy}{G}-1)$ (we will see in \cref{ObservationNonEmptyBelowEntry} that this position is always nonempty, by the definition of a height table). We can then look at the edge $zw$ in position $(z,\hh{xy}{G}-2)$, and so on, building a sequence of edges that goes ``downwards'' in the height table. As we will observe in \cref{ObservationLexOrderImpliesGraphOrder}, this sequence of edges is $\ordeqE$-increasing, but because vertices may appear multiple times in this sequence, we cannot guarantee that it gives an increasing path. Still, this kind of exploration going downwards in the height table is an important idea that will appear in our proof of \cref{thm:altitude-d}.

A crucial fact about height tables is that if we pass to a subgraph of $G$ (without changing the vertex or edge orderings) and rebuild the height table, then there are some edges of $G$ whose height does not decrease too much. Specifically, one of the main results of this section will be the following.

\begin{lemma}\label{LemmaEdgeDrop}
Let $G$ be an ordered graph, and let $S, T\subseteq E(G)$ with $|S|> |T|$. Then there is an edge 
$e\in S\setminus T$ with 
$$\hh e{G\setminus T}\geq \min_{f\in S} \hh{f}{G}.$$
\end{lemma}

We remark that this lemma has the same flavour as a result that appeared in~\cite{mil17}. It was shown there that if a small set of \emph{vertices} is deleted and the height table is rebuilt, then one can bound the decrease in height of \emph{every} edge. However, the result in~\cite{mil17} does not appear to be powerful enough for the approach in our paper, and our proof of \cref{LemmaEdgeDrop} seems to be quite different.

\cref{LemmaEdgeDrop} will be used to prove the following lemma, which shows that for any high edge $e$ and set of vertices $U$, there is a short increasing path starting from $e$ which ends at an edge which is high with respect to the height table of $G-U:=G[V(G)\setminus U]$. 

\begin{lemma}\label{LemmaLength3PathDrop}
Let $G$ be an ordered graph, and let $U\subseteq V(G)$. Consider an edge $xy \in E(G)$ with $\hh{xy}{G}> 4m+3$ for some real number $m$ satisfying $m\geq |U|$ and $m^2/2> \Delta(G)|U|$. 
Then there are vertices 
$z,w \in V(G)\setminus U$ such that $xyzw$ is an increasing path and 
$$\hh{zw}{G-U}\geq \hh{xy}{G}-4m-3.$$
\end{lemma}

We will prove \cref{LemmaLength3PathDrop} later in this section, but the rough idea is that we will greedily find a large number of length-3 increasing paths $xyzw$ extending $xy$, and we will apply \cref{LemmaEdgeDrop} with $S$ as the set of all such $zw$ and $T$ as the set of edges touching $U$. The purpose of \cref{LemmaLength3PathDrop} is that it will allow us to build a long increasing path in an iterative fashion, as follows. If we can find a reasonably long increasing path $P$ among the top few rows of the height table (ending at some edge $xy$), then by \cref{LemmaLength3PathDrop} we can find a continuation $xyzw$ of this path such that after deleting $P$, the edge $zw$ is still near the top of the height table, and it remains to find an increasing path from $zw$.

\mk{I moved the statements of \cref{LemmaEdgeDrop} and \cref{LemmaLength3PathDrop} to the start of the section so that the reader can immediately see what the point is. OK?}

Before proving \cref{LemmaEdgeDrop,LemmaLength3PathDrop}, we make a number of basic observations about height tables. First, the following observation allows us to go between the two orders $\ordE$ and $\llex$.
\begin{observation}\label{ObservationLexOrderImpliesGraphOrder}
Let $G$ be an ordered graph, and consider edges $e,f \in E(G)$ both containing the vertex $\vv f G$, and satsifying $\htf(f) \llex \htf(e)$. Then $e\ordE f$. \mk{I rephrased this statement, I thought it sounded pretty awkward before}
\end{observation}
\begin{proof}
When we chose to put $f$ at position $\htf(f)$, the edge $e$ was still available (as it got assigned to $\htf(e) \glex \htf(f)$). So, we must have had $e\ordE f.$
\end{proof}

The next lemma shows that every dense graph has a high edge. 
\begin{lemma}\label{LemmaHighEdgeInDenseGraph}
Every ordered graph $G$ has an edge $e$ with $\hh e G\geq  \d(G)/2$.
\end{lemma}
\begin{proof}
As every edge gets entered exactly once in $\HT(G)$ we have $|\{\htf(e): e\in E(G)\}|=|E(G)|$.
As there are $|V(G)|$ columns this means there needs to be a column with at least $|E(G)|/|V(G)|$ edges, so one of these edges has height at least $ |E(G)|/|V(G)| \ge\d(G)/2,$ as required. 
\end{proof}
\mk{I changed $d(G)$ to $\bar d(G)$ to make it completely clear that it means average degree, OK?}

The following two observations show that if some location in the height table is nonempty, then all the locations below it are non-empty as well. 
\begin{observation}\label{ObservationNonEmptyBelowEntry}
Let $G$ be an ordered graph, consider an edge $xy\in E(G)$. If $\htf(xy) \glex (i,x)$, then there is an edge $xz$ with $\htf(xz)=(i,x)$.
\end{observation}
\begin{proof}
Let us assume for the sake of contradiction that the entry $(i,x)$ is empty. This would mean that when we were constructing $\HT(G)$, and reached the position $(i,x)$, there were no remaining edges containing $x$. Specifically, $xy$ must have already been entered into the height table, meaning $(i,x) \glex \htf(xy)$, a contradiction.
\end{proof}

\begin{observation}\label{ObservationNonEmptyBelowRow}
Let $G$ be an ordered graph, and consider an edge $xy\in E(G)$. If $\hh{xy}G > i$, then there is an edge $xz$ with $\hh{xz}G = i$ and $\vv{xz}G=x$.
\end{observation}
\begin{proof}
Notice that $\htf(xy) \glex (i,x)$ by the definition of $\glex$. The observation follows from \cref{ObservationNonEmptyBelowEntry}.
\end{proof}

The following lemma shows that if we pass to a spanning subgraph of $G$, and we rebuild the height table, then no edge can increase in height.
\begin{lemma}\label{LemmaSubgraphMonotonicity}
Let $G$ be an ordered graph and let $H$ be a spanning subgraph of $G$. Then for any edge $e\in E(H)$ we have $\hth(e) \leqlex \htf(e)$.
\end{lemma}
\begin{proof}
Suppose, for the sake of contradiction, that the lemma is false, and let $e\in E(H)$ be the $\ordeqE$-largest edge with $\hth(e) \glex \htf(e)$.
Let $\vv e G=v$ and $e=vy$. 
By 
\cref{ObservationNonEmptyBelowEntry} applied to $H$ (with $i=\hh e G$, $x=v$ and $xy=e$), there is an edge $f\in E(H)$ with $\hth(f)=\htf(e)=(\hh{vy}G, v)$.

Since  $\vv f H=v=f\cap e$ and $\hth(f)\llex \hth(e)$, \cref{ObservationLexOrderImpliesGraphOrder} implies that $e\ordE f$.
Since $\vv e G=v=f\cap e$ and $e<_G f$, the contrapositive of \cref{ObservationLexOrderImpliesGraphOrder} implies that $\htf(e)\glex \htf(f)$. 
Thus, we have $e<_G f$ and $\hth(f)=\htf(e)\glex \htf(f)$, contradicting the maximality of $e$. 
\end{proof}

We are now ready to prove \cref{LemmaEdgeDrop}.

\begin{proof}[Proof of \cref{LemmaEdgeDrop}]
Construct an auxiliary digraph $D$ on $\mathbb{N}\times V(G)$ by placing directed edges from $\htf(e)$ to $\htt(e)$ for all $e \in G\setminus T$. Delete all loops in $D$.

Since every edge appears at most once in $\HT(G)$ and $\HT(G \setminus T)$, each vertex in $D$ has in-degree and out-degree at most 1. This implies that $D$ is a union of vertex-disjoint directed paths and cycles.

\cref{LemmaSubgraphMonotonicity} applied with $H = G\setminus T$ implies that if $((i,x),(j,y))$ is a directed  edge of $D$, then $(j,y)\llex (i,x)$. By transitivity of $\llex$ this implies that 
\begin{equation}\label{EqPathMonotonicity}
\text{If there is a directed path $P$ from $(i,x)$ to $(j,y)$ in $D$ then $(j,y)\llex (i,x)$.}
\end{equation}
In particular, this shows that $D$ is acyclic and so a union of directed paths. The following claim further characterises these paths.

\begin{claim}
If $d^-((i,x))=1$ for $(i,x)\in \mathbb{N}\times V(G)$ then either $d^+((i,x))=1$ or $(i,x)=\htf(t)$ for some $t\in T$.
\end{claim}
\mk{would this read better if it was just part of the proof, not a nested claim?}
\begin{proof}
Suppose that we have a directed edge from $\htf(e)$ to some $(i,x)=\htt(e)\neq \htf(e)$ in $D$. 
Then we have $(i,x)\llex \htf(e)$ by \cref{LemmaSubgraphMonotonicity}. By \cref{ObservationNonEmptyBelowEntry} there is an edge $f\in E(G)$ with $(i,x)= \htf(f)$ (for this application of \cref{ObservationNonEmptyBelowEntry}, we use $x\in e$ coming from $(i,x)=\htt(e)$). If $f\in T$ then we are done.
Otherwise, $\htt(f)$ is well-defined. Note that $e \neq f$, because $\htf(f)=(i,x) \neq \htf(e)$. Therefore, we have $\htt(f)\neq \htt(e)=(i,x)$. It follows that $d^+((i,x))=d^+(\htf(f))=1$.
\end{proof}

We now return to the proof of \cref{LemmaEdgeDrop}. From the above claim it follows that the number of vertices $(i,x)$ in $D$ with $d^-((i,x))=1$ and $d^+((i,x))=0$ is at most $|T|$. Since every non-trivial path in $D$ ends in such a vertex, we conclude that $D$ is a path forest having at most $|T|$ paths which are either non-trivial or consist of a single vertex of $T$.
 
Suppose that there is a directed path from $\htf(e)$ to  $\htf(f)$ for distinct $e, f\in S.$ Then $\htt(e)$ is the second vertex on this path so $\htt(e) \geqlex \htf(f)$ by  \cref{EqPathMonotonicity}. By the definition of $\glex$ this implies $\hh{e}{G \setminus T} \geq \hh{f}{G} \ge \min_{g\in S} \hh{g}{G},$ so $e$ satisfies the desired condition. Otherwise, every path in $D$ contains at most one $\htf(e)$, for $e\in S$. Now, if $\htf(e)=\htt(e)$ for some $e \in S \setminus T$, then this $e$ satisfies the requirement of the lemma. Otherwise, each $e\in S \setminus T$ lies on some non-trivial path in $D$. This shows that there are at least $|S|$ paths in $D$ which are non-trivial or consist of a vertex of $T$, contradicting $|T|<|S|$.
\end{proof}

To prove \cref{LemmaLength3PathDrop}, we introduce another definition, regarding the possible vertices with which we can extend an increasing path. This definition will also be used later in the paper.
\begin{definition}
\label{def:Si}
Given an edge $xy$ of an ordered graph $G$, and any $i<\hh{xy}G$, we denote by $S_i(x,y)$ the set of vertices $z$ such that 
$\vv{yz}{G}=y$ and $\hh{xy}{G}-i \le \hh{yz}{G}<\hh{xy}{G}.$ (This set does not contain $x$ or $y$).
\end{definition}
Note that \cref{ObservationNonEmptyBelowRow} implies $|S_i(x,y)|=i$ and \cref{ObservationLexOrderImpliesGraphOrder} implies that for any $z\in S_i(x,y)$ the path $xyz$ is increasing. Now, we prove \cref{LemmaLength3PathDrop}.

\begin{proof}[Proof of \cref{LemmaLength3PathDrop}]
Let $Z=S_{\ceil{2m}}(x,y) \setminus U,$ so that $|Z|\ge 2m-m=m.$ For any $z \in Z$ let $W_z=S_{\ceil{2m+1}}(y,z) \setminus (U \cup \{x\}),$ and note that $|W_z|\ge m$ for all $z \in Z$. Here we need the fact that $\hh{xy}{G}> 4m+3$ so that $S_{\ceil{2m}}(x,y)$ and $S_{\ceil{2m+1}}(y,z)$ are well-defined.

Let $S=\{zw: z\in Z,w \in W_z\}$, which is a set of at least $m^2/2$ edges disjoint from $U$ (we divide by 2 because an edge $zw\in S$ may arise from both $z\in Z,w\in W_z$ and from $w\in Z,z\in W_w$). Let $T$ be the set of all edges in $G$ touching $U$. Notice that we have $|S|\geq m^2/2>\Delta(G)|U|\geq |T|$. By \cref{LemmaEdgeDrop} there is an edge $zw\in S$ with 
$\hh{zw}{G\setminus T}\geq \min_{f\in S} \hh{f}{G}.$ This implies 
$$\hh{zw}{G-U}= \hh{zw}{G\setminus T}\geq \min_{f\in S} \hh{f}{G}\geq \hh{xy}{G}-4m-3.$$
By construction the path $xyzw$ is increasing, so satisfies the lemma.
\end{proof}

\section{Finding long increasing paths}
\label{sec:main-proof}

In this section we combine the tools developed in \cref{sec:regularisation,sec:height-table} to prove \cref{thm:altitude-d}, which implies \cref{thm:altitude-complete}. Actually, \cref{thm:altitude-d} will be a consequence of the following stronger result.

\begin{theorem}\label{thm:1-eps-regular}
The following holds for any integer $t\ge 1$. Let $G$ be an ordered graph with $n\ge 2$ vertices and 
consider an edge $e\in E(G)$ with $\hh{e}{G}> a$, for some real number $a>0$. Then there is an increasing path $P$ in $G$ starting with $e$
and having length at least
$$\frac{a^{1-1/t}}{(70\log n)^{2t}},$$
such that $\hh{f}{G}\ge \hh{e}{G}-a$ for every $f\in E(P)$.
\end{theorem}

Before proceeding to the proof of \cref{thm:1-eps-regular}, we briefly show how \cref{thm:altitude-d} may be deduced from it.

\begin{proof}[Proof of \cref{thm:altitude-d}]
Choose $a=\Omega(d)$ with $a<d/2$. By \cref{LemmaHighEdgeInDenseGraph}, there is an edge $e$ with $\hh{e}{G}> a$. Then, apply \cref{thm:1-eps-regular} with $t=\floor{\sqrt{\frac{\log a}{\log \log n}}}$. The desired result follows, noting that $a^{1/t},(70\log n)^{2t}=2^{O\left(\sqrt{\log d\log \log n}\right)}$.
\end{proof}

\mk{there used to be a more general (and more technical) corollary of \cref{thm:1-eps-regular} here instead, with explicit constants. There didn't seem to be any reason for it, so I removed it? Am I missing something?}

We will prove \cref{thm:1-eps-regular} by induction. Roughly speaking, for some suitable $m$, the idea is to use the induction hypothesis to find a reasonably long increasing path starting from $e$ in the $m$ rows of the height table just below $e$, then to delete the vertices in this path and use \cref{LemmaLength3PathDrop} to show that there is a potential continuation of our increasing path that is still near the top of the height table (the optimal value of $m$ is determined by the tradeoff between the induction hypothesis and \cref{LemmaLength3PathDrop}). We can then repeat this argument about $a/m$ times. An issue with this plan is that \cref{LemmaLength3PathDrop} is not effective for graphs with very large maximum degree, but we have no control over the maximum degree of the graphs that arise during our proof. We therefore use \cref{lem:regularisation} to find a subgraph whose maximum degree is not too large, and work inside this subgraph. Unfortunately, passing to this subgraph may involve deleting $e$ and edges incident to it, so we will need some lemmas to ensure that our final increasing path can be connected to $e$.

An \emph{increasing trail} in an edge-ordered graph is a trail (possibly with repeated vertices) with a specified direction, whose edges form an increasing sequence. Given a sufficiently high edge of an ordered graph $G$, the following lemma shows how to find a subgraph of $G$ with large average degree such that every edge of this subgraph can be reached with a short increasing trail starting at $e.$

\begin{lemma}\label{lemma:dense-graph-in neighbourhoods}
Let $G$ be an ordered graph with 
$|V(G)| \le n$, and let $e$ be an edge of $G$ with $\hh{e}{G}\ge 21 h \log n,$ for some 
real number $h\ge 1.$ Then one can find a subgraph $H \subseteq G$ with average degree $\d(H)\ge h,$ such that for each $f \in E(H)$ there is an increasing trail $T$ in $G$ with the following properties:
\begin{enumerate}
    \item $T$ starts with $e$ and ends with $f,$
    \item $T$ has length at most $2+\log n,$
    \item $\hh{g}{G} \ge \hh{e}{G}-7 h(\log n+2)$ for every edge $g\in E(T).$
\end{enumerate}
\end{lemma}

\begin{proof}
Recall the definition of $S_i(x,y)$ from \cref{def:Si}.

Let $e=x_0x_1$ with $\vv{e}{G}=x_0.$ We say a trail $x_0x_1\ldots x_i$ is \textit{controlled} if it is an increasing trail with 
$\hh{x_{j-1}x_{j}}{G} \ge \hh{e}G-7hj$ and $\vv{x_{j-1}x_j}{G}=x_{j-1}$ for every $j\in \{1,\ldots,i\}$. We define $N_i$ to be the set of vertices $x_i$ for which there is a controlled trail of length $i$ ending in $x_i.$ 
Fix an integer $i$ with $1 \le i \le 1+\log n$, and let $x_i$ be an arbitrary vertex in $N_i$. By the definition of $N_i$, there is a controlled trail of the form $W=x_0x_1\ldots x_{i-1}x_i$. We have $\hh{x_{i-1}x_i}{G}\ge \hh{e}{G}-7hi \ge 21 h\log n -7hi \ge 7h$ for $i \le 1+\log n$, so $S_{\floor{7h}}(x_{i-1},x_i)$ is well-defined. Observe that any vertex $x_{i+1}\in S_{\floor{7h}}(x_{i-1},x_i)$ extends $W$ to a controlled trail ending in $x_{i+1},$ implying $x_{i+1} \in N_{i+1}$. 
Note that there are $|S_{\floor{7h}}(x_{i-1},x_i)|=\floor{7h} \geq 6h$ distinct  choices for such $x_{i+1}$.

Let $k$ be the smallest index such that $|N_{k+1}|\leq 2|N_k|.$ 
As $n \ge |N_k| \ge 2^{k-1}|N_1|=2^{k-1}$, we must have $k \le 1+\log n$. Consider the subgraph $H\subseteq G$, whose vertex set is $N_k \cup N_{k+1}$ and whose edge set consists of all edges at the end of a controlled trail of length $k+1.$ The observation in the above paragraph implies that every vertex in $N_k$ has degree at least $6h$ in $H.$ So, $H$ has at most $3|N_k|$ vertices and at least $|N_k|\cdot 6h/2$ edges (note that $N_k$ and $N_{k+1}$ might not be disjoint), meaning that it has average degree at least $2h\ge h.$
\end{proof}

The following lemma will be used in combination with \cref{lemma:dense-graph-in neighbourhoods}. It says that if we have a short increasing trail between edges $e$ and $f$, and we have a long increasing path starting with $f$, then we can combine these to find an increasing path starting with $e$ that is still quite long.
\begin{lemma}\label{lemma:join-trail-path}
Given an increasing trail $W=w_0\dots w_k$ and an increasing path $P=p_0\dots p_\ell$ with $w_{k-1}w_k=p_0p_1,$ we can obtain an increasing path using only edges of $E(W) \cup E(P)$ starting with $w_0w_1$ and having length at least $\ell/(k+1)-1.$ 
\end{lemma}
\begin{proof}
We distinguish two cases: either $w_{k-1}=p_0$ and $w_k=p_1$, or $w_{k-1}=p_1$ and $w_k=p_0.$ Let us first consider the former case. For each vertex $w_i\in V(W)\cap V(P)$, consider the path $P_{w_i}$ starting at $w_i$ and continuing along $P$ as long as possible before reaching another vertex of $V(W)$. Since $P$ starts with a vertex of $V(W)$, the paths $P_{w_i}$ partition the vertex set of $P$. As $|V(W)|=k+1$, there are at most $k+1$ of these subpaths so one of them (say $P_{w_i}$) needs to contain at least $(\ell+1)/(k+1)$ vertices. Now, consider a minimal trail, among the edges of $W$, starting with $w_0$ or $w_1$ and ending in $w_i.$ By the minimality assumption it must be a path and contain only one of $w_0,w_1$. Hence, appending to it the edge $w_0w_1$, we obtain an increasing path $W'$ between $w_0w_1$ and $w_i$. Then, concatenating $W'$ and $P_{w_i}$ yields an increasing path of length at least $(\ell+1)/(k+1)-1$.

Returning to the latter case where $w_{k-1}=p_1$ and $w_k=p_0$, we can consider the trail $W'=w_0\dots w_{k-1}p_2$ and the path $P'=p_1\dots p_\ell$. We are now in the situation of the former case, but with a path of length $\ell-1$. Repeating the above argument yields the desired result.
\end{proof}
\mk{I moved the above two lemmas to this section because the last section is more suited to basic lemmas about height tables.}

We are now finally ready to prove \cref{thm:1-eps-regular}.

\begin{proof}[Proof of \cref{thm:1-eps-regular}]
Let 
$n=|V(G)|$ and $d=\d(G).$ Let $C=70.$ 
The proof will proceed by induction on $t$. 
The base case $t=1$ holds trivially. We now assume the claim is true for some choice of 
integer $t\ge 1$, and show that it is true for $t+1.$

First note that we may assume
$$\frac{a^{1-1/(t+1)}}{(C \log n)^{2(t+1)}}>1,$$ as otherwise the desired result holds trivially. It follows that
\begin{align}\label{equ:a-big}
a>(C\log n)^{2t+4}.    
\end{align}
That is, we can assume that $a$ is large for the rest of the proof. Now, we apply \cref{lemma:dense-graph-in neighbourhoods} with $G$ and $e$, taking 
$\d(G')\ge a/(21\log n)$, such that for every edge $f\in E(G')$ there is a trail $T_f$ in $G$ with the following properties:
\begin{enumerate}
    \item $T_f$ starts with $e$ and ends with $f$,
    \item $T_f$ has length at most $2+\log n$, and
    \item $\hh{g}G \ge \hh{e}G-a$ for every $g\in E(T_f)$.
\end{enumerate}
Note that by removing some edges if necessary we can assume  
$\d(G') = \floor{a/(21\log n)}.$ Next, we apply \cref{lem:regularisation} to $G'$ to obtain a subgraph $H$ of $G'$ with 
\begin{align}\label{equ:0}
    \d(H) \geq \frac{4a}{(C\log n)^2}, & \quad\Delta(H) < \frac{120a}{(C\log n)^2}
\end{align}
(we have used \cref{equ:a-big} to simplify the above inequalities and to deduce $|V(G')| \ge 2$). We now proceed to find a long increasing path $P,$ within $H,$ by iteratively invoking the inductive assumption together with \cref{LemmaLength3PathDrop}. After we have done this, the final step of the proof will be to use \cref{lemma:join-trail-path} to combine $P$ with some $T_f$, yielding an increasing path starting with $e$ with the desired properties.

Define 
$$m=\frac{(240a)^{1-1/(t+1)}}{(C\log n)^{2t}},\quad\ell=\frac{m^{1-1/t}}{(C\log n)^{2t}}=\frac{(240a)^{1-2/(t+1)}}{(C\log n)^{4t-2}}.$$
The parameter $m$ will control how many rows of the height table we will use per application of the inductive hypothesis, while $\ell$ will denote the length of the path with which we extend $P$ in each iteration.

\begin{claim}\label{claim:iteration}
There is a sequence of graphs $H=G_1\supseteq G_2\supseteq \dots\supseteq G_{\ceil{\d(H)/(48m)}}$ and a sequence of paths $P_1,\dots, P_{\ceil{\d(H)/(48m)}}$, with $P_i \subseteq G_i$, satisfying the following properties for each $i$. (Let $e_i=x_iy_i$ be the starting edge and let $f_i=z_iw_i$ be the ending edge of each $P_i$).

\begin{enumerate}[label=(\alph*)]
    \item \label{enum:a} $G_{i+1}=G_i \setminus V(P_i-\{z_i,w_i\}),$
    \item \label{enum:b} 
    $|P_i| = \ceil{\ell}$,
    \item \label{enum:c} 
    $z_iw_ix_{i+1}y_{i+1}$ is an increasing path, and
    \item \label{enum:d} 
    $\hh{f_i}{G_i}\ge \d(H)/2-6im$.
\end{enumerate}
\end{claim}
\begin{proof}

 Let us start by observing several inequalities we will require in order to be able to apply \cref{LemmaLength3PathDrop}. Using \cref{equ:0}, we have  
 \begin{align}\label{equ:1}
 m^2=\frac{240a\ell}{(C\log n)^2}>2\Delta(H)\ell.
 \end{align}
 Furthermore, $a/m>(C \log n)^{2t+2}/240>7(C\log n)^2$ by \cref{equ:a-big}, giving 
 \begin{align}\label{equ:2}
     \d(H)/4 & \ge \frac{a}{(C\log n)^2} > 7m \ge 4m+3.
 \end{align}
 Finally, notice that 
 \begin{align}\label{equ:3}
 \ceil{\ell}&<m.
 \end{align}
 
We will define the $G_i$ and $P_i$ inductively. Choose $e_1 \in E(H)$ with $\hh{e_1}{H} \ge \d(H)/2.$ This is possible by \cref{LemmaHighEdgeInDenseGraph}. We start by applying the inductive assumption to $H$ and $e_1$ with $a=m$ (recall that $\d(H)/2>m$ by \cref{equ:2}). This yields a path $P_1$ of length $\ceil{\ell}$ 
  starting with $e_1$ and ending with an edge $f_1$ satisfying $\hh{f_1}{G} \ge \d(H)/2-m.$

Next, consider $1 \le i \le \d(H)/(48m)-1$, and suppose $G_1,\dots,G_i$ and $P_1,\dots,P_i$ have already been constructed. Define $G_{i+1}$ according to \ref{enum:a}. We then apply \cref{LemmaLength3PathDrop} to $G_i$ (taking $U=V(P_i)\setminus \{z_i,w_i\}$) to find an edge $e_{i+1}=x_{i+1}y_{i+1}$ such that \ref{enum:c} is satisfied and $\hh{e_{i+1}}{G_{i+1}} \ge \hh{f_i}{G_i}-5m$. Equations~\eqref{equ:1} to~\eqref{equ:3} ensure that the conditions of \cref{LemmaLength3PathDrop} are satisfied.
 
 Since $i+1 \le \ceil{\d(H)/(48m)}$, we can use \ref{enum:d} and the bound $\d(H)>28m$ from \cref{equ:2} to deduce that $\hh{f_i}{G_i}\ge \d(H)/2-6im \ge \d(H)/4+6m$. It follows that $\hh{e_{i+1}}{G_{i+1}}\ge \d(H)/4>4m,$ using \cref{equ:2}. Thus we can apply the inductive assumption to $G_{i+1}$ with starting edge $e_{i+1}$ and $a=m$. This gives a path $P_{i+1}$ of length $\ceil{\ell}$ 
 starting with $e_{i+1}$ and ending with some $f_{i+1}$ satisfying $\hh{f_{i+1}}{G_{i+1}} \ge \hh{e_{i+1}}{G_{i+1}}-m\ge \hh{f_i}{G_i}-6m$. (This final inequality proves that  \ref{enum:d} is satisfied).
\end{proof}

We now return to the proof of \cref{thm:1-eps-regular}. Note that the paths produced by \cref{claim:iteration} are all disjoint. Condition \ref{enum:c} allows us to join them up into one path $P'$ starting at $e_1=x_1y_1$ and of length at least
\begin{align*}
\frac{\d(H)}{48m} \cdot \ell & \ge \frac{a^{1-1/(t+1)}}{12\cdot 240^{1/(t+1)}\cdot (C \log n)^{2t}} \\
&\ge 8\log n \cdot \frac{a^{1-1/(t+1)}}{(C \log n)^{2t+2}}.
\end{align*}
(In the first inequality we used the estimate $\d(H)\ge \frac{4a}{(C\log n)^2}$ from \cref{equ:0}). Note that when joining the paths, the path $P_{i+1}$ might start with the vertex $y_{i+1}$, in which case we simply replace the edge $y_{i+1}x_{i+1}$ with $w_ix_{i+1}$.

Finally, we apply \cref{lemma:join-trail-path} to join the trail $T_{e_1}$ with the path $P'$ and obtain an increasing path starting at $e$ of length at least
$$\frac{a^{1-1/(t+1)}}{(C \log n)^{2(t+1)}}.$$
As this path lies completely within 
$T_{e_1}\cup G',$ its edges have heights at least $\hh{e}{G}-a,$ as desired.
\end{proof}

\section{Concluding remarks}
\label{sec:concluding}
In this paper we have proved that any edge-ordering of the complete graph on $n$ vertices contains a monotone path of length $n^{1-o(1)}$. We also proved more generally that if $d=(\log n)^{\omega(1)}$ then in any edge-ordering of any $n$-vertex graph with average degree $d$, there is a monotone path of length $d^{1-o(1)}$. Of course, there is still room for improvement in both these results. Does any edge-ordering of the complete graph permit a monotone path of length $\Omega(n)$, or even $(1/2-o(1))n$ (as asked by Calderbank, Chung and Sturtevant~\cite{CCS84})? Can one improve R\"odl's bound of $\Omega(\sqrt d)$ for graphs with average degree $d$, when $d$ is very small relative to $n$? We observe that R\"odl's bound can indeed be improved for graphs which are locally sparse.

\begin{proposition}
\label{prop:locally-sparse}
Let $G$ be an edge-ordered graph with average degree $d$, and suppose there is $0<\varepsilon<1$ such that every set of at most $\varepsilon d$ vertices induces at most $(1/2-\varepsilon)d$ edges. Then $G$ has a monotone path of length $\varepsilon d$.
\end{proposition}
\begin{proof}
We use the machinery developed in \cref{sec:height-table}. Fix a vertex-ordering of $G$ and consider its height table. By \cref{LemmaHighEdgeInDenseGraph}, there is an edge $x_0x_1$ with height $\hh{x_0x_1}G\ge d/2$. Assume without loss of generality that $\vv{x_{0}x_{1}}G=x_{0}$. Now, we will iteratively build an increasing path $x_0x_1\dots x_{\varepsilon d}$, with each $\vv{x_{i-1}x_{i}}G=x_{i-1}$. Given $x_0x_1\dots x_{i-1}$, we show how to choose $x_i$. For each $1\le h<\hh{x_{i-2}x_{i-1}}G$, by \cref{ObservationNonEmptyBelowRow} there is an edge $x_{i-1}y_h$ in position $(x_{i-1},h)$ of the height table (also, we have $x_{i-1}y_h\ordeqE e$ by \cref{ObservationLexOrderImpliesGraphOrder}). Let 
$$h^*=\max\left\{h:1\le h<\hh{x_{i-2}x_{i-1}}G,y_h\notin \{x_0,\dots,x_{i-1}\}\right\},$$
and let $x_i=y_{h^*}$. That is, we consider the position $(x_{i-1},h)$ in the height table and scan through all edges vertically below it, searching for the first suitable edge to extend our path.

Suppose this procedure fails to produce a path of length $\varepsilon d$ (say it produces a path $P=x_0\dots x_{\ell}$ of length $\ell<\varepsilon d$). Then, during the above procedure there were more than $(1/2-\varepsilon)d$ instances where we looked at an edge but could not add it because it was between two vertices in $P$. But this is impossible, because our local sparsity condition implies that the vertex set of $P$ induces at most $(1/2-\varepsilon)d$ edges.
\end{proof}

In particular, we remark that \cref{prop:locally-sparse} can be used to prove that if $p=n^{-1/2-\varepsilon}$, for fixed $\varepsilon>0$, then any edge-ordering of the random graph $G\in \mathbb G(n,p)$ typically has a monotone path of length about 
$\varepsilon np$, which is proportional to its average degree. In this regime, it gives a tight result and improves the lower bound $\frac{np}{\omega(1)\log n}$ due to De Silva, Molla, Pfender, Retter and Tait~\cite[Theorem 6]{DMPRT16}.

Finally, we remark that Chung and Graham (see~\cite{CG}) proposed the following general question: letting $p(v)$ denote the maximum length of an increasing path starting at vertex $v$, is it always true that 
$\sum_v p(v)\geq|E(G)|$?

\subsubsection*{Acknowledgements}
The third author thanks Katherine Staden and Tibor Szab\'o for stimulating discussions related to this paper. We would also like to thank the referee for their careful reading of the manuscript and their valuable comments.

\begin{appendices}
\crefalias{section}{appsec}

\section{Graphs with no large nearly-regular subgraph}
\label{sec:no-regular-subgraph}
In this appendix we prove that \cref{lem:regularisation} is essentially optimal, using a variant of a probabilistic construction of Pyber, R\"odl and
Szemer\'edi~\cite{PRS95} (see also~\cite{Alo08}).
\begin{proposition}
For any fixed $0<\varepsilon<1$ and $K\in\NN$, and sufficiently
large $N$ (relative to $\varepsilon,K$), there is an $N$-vertex
graph $F$ with average degree at least $N^{\varepsilon}$, satisfying
the following property. If $H$ is a subgraph of $F$ with all degrees
lying in the range $\left[d',Kd'\right]$, then $d'\le\left(64K^{3}\right)N^{\varepsilon}/\floor{\left(1-\varepsilon\right)\log N-2}$.
\end{proposition}

\begin{proof}
We will define a random graph $G$ parametrised by an integer $n$,
which will have between $n$ and $2n$ vertices. We will show that
with probability $1-o\left(1\right)$, this graph $G$ has average degree at least
$4n^{\varepsilon}$, and that if $H$ is a subgraph of $G$ with degrees
lying in the range $\left[d',Kd'\right]$, then 
$d'\le\left(64K^{3}\right)n^{\varepsilon}/\floor{\left(1-\varepsilon\right)\log n-1}$.

For sufficiently large $N$ we may then obtain the desired graph $F$
by taking a typical outcome of $G$, for $n=\ceil{N/2}$, and adding
at most $N/2$ isolated vertices.

Let $\ell=\floor{\left(1-\varepsilon\right)\log n-1}$,
and let $A,B_{1},\dots,B_{\ell}$ be disjoint sets with $\left|A\right|=n$
and $\left|B_{i}\right|=2^{i}\floor{n^{\varepsilon}}$. Let $B:=\bigcup_{i=1}^{\ell}B_{i}$,
which has size at most $n$. Consider the random bipartite graph $G$
with parts $A$ and $B:=\bigcup_{i=1}^{\ell}B_{i}$, where each edge
between $A$ and $B_{i}$ is present with probability $4/\left(2^{i}\ell\right)$. For an integer $r$, let $B_{<r}=\bigcup_{i=1}^{r-1}B_{i}$ and $B_{\ge r}=B\backslash B_{<r}=\bigcup_{i=r}^{\ell}B_{i}$ (if $r\le 1$ then $B_{<r}=\emptyset$ and $B_{\ge r}=B$).
We claim that the following properties hold simultaneously with positive
probability.
\begin{itemize}
\item $G$ has at least $2n^{1+\varepsilon}$ edges;
\item For any $m\ge n^{\varepsilon/2}$, with $r=\ceil{\log \left(m/\left(4K^{2}n^{\varepsilon}\right)\right)}$,
if $A'\subseteq A$ and $B'\subseteq\bigcup_{i=r+1}^{\ell}B_{i}$ satisfy
with $\left|A'\right|,\left|B'\right|\le m$,
then $e\left(A',B'\right)\le\left(32K^{2}n^{\varepsilon}\right)m/\ell$.
\end{itemize}
Both these properties can easily be shown to hold with probability
$1-o\left(1\right)$ with a large deviation inequality. Before we
give the details, we show how these properties imply the desired result.
So, suppose $G$ satisfies the above properties. By the first property,
$G$ has average degree at least $4n^{\varepsilon}$. Then, consider
a subgraph $H$ of $G$ with degrees lying in the range $\left[d',Kd'\right]$.
We wish to show that $d'\le\left(64K^{3}\right)n^{\varepsilon}/\ell$.

Let $A\left(H\right)=A\cap V\left(H\right)$ and $B\left(H\right)=B\cap V\left(H\right)$,
and let $m=K\left|A\left(H\right)\right|$. We must have $m/K^{2}\le\left|B\left(H\right)\right|\le m$,
because 
$\sum_{a\in A(H)}\deg_{H}\left(a\right)=\sum_{b\in B(H)}\deg_{H}\left(b\right)$
.
Note that 
\begin{equation}
e_{H}\left(A\left(H\right),B\left(H\right)\right)\ge d'\left|A(H)\right|\ge d'm/K.\label{eq:pyber-lower}
\end{equation}
Note also that we can assume $m\ge n^{\varepsilon/2}$ (otherwise
each vertex in $B$ trivially has degree at most $\left|A(H)\right|\le m<Kn^{\varepsilon/2}<\left(64K^{3}\right)n^{\varepsilon}/\ell$,
and we are done). Now, let $r=\ceil{\log \left(m/\left(4K^{2}n^{\varepsilon}\right)\right)}$, so $\left|B_{<r}\right|\le 2^r n^\epsilon\le m/\left(2K^{2}\right)$ (if $r\le 1$ this is trivially true, otherwise it follows from the fact that $|B_{<r}|\le |2B_{r-1}|$). We have
\begin{equation}
e_{H}\left(A(H),B\left(H\right)\cap B_{<r}\right)\le Kd'\cdot m/\left(2K^{2}\right)=d'm/\left(2K\right).\label{eq:pyber-upper-small}
\end{equation}
Next, let $B_{\ge r}=B\backslash B_{<r}=\bigcup_{i=r}^{\ell}B_{i}$,
so by the second property of $G$, 
\begin{equation}
e_{H}\left(A\left(H\right),B\left(H\right)\cap B_{\ge r}\right)\le\left(32K^{2}n^{\varepsilon}\right)m/\ell.\label{eq:pyber-upper-big}
\end{equation}
Combining \cref{eq:pyber-lower}, \cref{eq:pyber-upper-small} and \cref{eq:pyber-upper-big}
gives
\[
d'm/\left(2K\right)\le\left(32n^{\varepsilon}K^{2}\right)m/\ell,
\]
meaning that $d'\le\left(64K^{3}\right)n^{\varepsilon}/\ell$, as
desired.

Now we prove the claimed properties of $G$. For $v\in A$ and $w\in B_i$
let $X_{v,w}=\one_{\left\{ v,w\right\} \in E\left(G\right)}$, so
that $\E X_{v,w}=4/\left(2^{i}\ell\right)$ and $e\left(G\right)=\sum_{\left(v,w\right)\in A\times B}X_{v,w}$.
We have 
\[
\E e\left(G\right)=\sum_{i=1}^{\ell}\sum_{v\in A}\sum_{w\in B_{i}}\frac{4}{2^{i}\ell}=4n \floor{n^{\varepsilon}},\quad\Var e\left(G\right)=\sum_{i=1}^{\ell}\sum_{v\in A}\sum_{w\in B_{i}}\frac{4}{2^{i}\ell}\left(1-\frac{4}{2^{i}\ell}\right)\le4n^{1+\varepsilon},
\]
so by Chebyshev's inequality we have $\Pr\left(e\left(G\right)<2n^{1+\varepsilon}\right)=o\left(1\right)$.
Similarly, for any $A'\subseteq A$ and $B'\subseteq\bigcup_{i=r}^{\ell}B_{i}$
with $\left|A'\right|,\left|B'\right|\le m$, we have $e_{G}\left(A',B'\right)=\sum_{\left(v,w\right)\in A'\times B'}X_{v,w}$.
Then,
\[
\E e\left(A',B'\right),\quad\Var e\left(A',B'\right)\le\frac{4}{2^{r}\ell}m^{2}\le\frac{16K^{2}n^{\varepsilon}m}{\ell},
\]
so by Bernstein's inequality (see for example~\cite{BLM13}), we have
\begin{align*}
\Pr\left(e_{G}\left(A',B'\right)>\frac{32K^{2}n^{\varepsilon}m}{\ell}\right)&\le
\Pr\left(e_{G}\left(A',B'\right)>\E e_{G}\left(A',B'\right)+\frac{16K^{2}n^{\varepsilon}m}{\ell}\right)\\
&\le\exp\left(-\frac{\frac{1}{2}\left(16K^{2}n^{\varepsilon}m/\ell\right)^{2}}{\left(16K^{2}n^{\varepsilon}m/\ell\right)+\frac{1}{3}\left(16K^{2}n^{\varepsilon}m/\ell\right)}\right)\\
&=e^{-\Omega\left(n^{\varepsilon}m/\ell\right)}=o\left(n^{-(2m+1)}\right).
\end{align*}
We may then take the union bound over all (at most $n$) choices of
$m$, and all (at most $\left(\binom{n}{1}+\dots+\binom{n}{m}\right)^2\le n^{2m}$) choices of $A',B'$.
\end{proof}

\end{appendices}

\end{document}